\documentclass[11pt]{article}
\usepackage{latexsym}
\usepackage{amsfonts}
\usepackage{mathrsfs,amsthm,amsmath}
\usepackage{color}
\newtheorem{theorem}{Theorem}[section]
\newtheorem{proposition}[theorem]{Proposition}

\newtheorem{claim}{Claim}[section]
\newtheorem{remark}{Remark}[section]

\newtheorem{condition}{Condition}[section]
\begin{document}
\title{Non-central moderate deviations for compound fractional Poisson processes\thanks{We
acknowledge the support of Indam-GNAMPA (for the Research Project \lq\lq Stime asintotiche:
principi di invarianza e grandi deviazioni\rq\rq), MIUR (for the Excellence Department 
Project awarded to the Department of Mathematics, University of Rome Tor Vergata (CUP 
E83C18000100006)) and University of Rome Tor Vergata (for the Research Program \lq\lq Beyond
Borders\rq\rq, Project \lq\lq Asymptotic Methods in Probability\rq\rq (CUP E89C20000680005)).}}
\author{Luisa Beghin\thanks{Address: Dipartimento di Scienze Statistiche, Sapienza 
Universit\`a di Roma, Piazzale Aldo Moro 5, I-00185 Roma, Italy. e-mail:
\texttt{luisa.beghin@uniroma1.it}}
\and Claudio Macci\thanks{Address: Dipartimento di Matematica, Universit\`a di Roma Tor 
Vergata, Via della Ricerca Scientifica, I-00133 Rome, Italy. e-mail:
\texttt{macci@mat.uniroma2.it}}}
\date{}
\maketitle
\begin{abstract}
The term \emph{moderate deviations} is often used in the literature to mean a class of large 
deviation principles that, in some sense, fill the gap between a convergence in 
probability to zero (governed by a large deviation principle) and a weak convergence to a
centered Normal distribution. We talk about \emph{non-central moderate deviations} when 
the weak convergence is towards a non-Gaussian distribution. In this paper we study non-central
moderate deviations for compound fractional Poisson processes with light-tailed jumps.
\end{abstract}

\section{Introduction}
The theory of large deviations gives an asymptotic computation of small probabilities on 
exponential scale (see \cite{DemboZeitouni} as a reference of this topic) and the basic definition
of this theory is the \emph{large deviation principle}. A large deviation principle provides some 
asymptotic bounds for a family of probability measures on the same topological space; these bounds 
are expressed in terms of a \emph{speed function} (that tends to infinity) and a nonnegative
lower semicontinuous \emph{rate function} defined on the topological space.

The term \emph{moderate deviations} is used for a class of large deviation principles which fill the
gap between a convergence to a constant (at least in probability) and governed by a large deviation 
principle, and an asymptotic normality result. A more precise description is given in the following 
claim, where $t\to\infty$.

\begin{claim}\label{claim:standardMD}
We have a family of $\mathbb{R}^h$-valued random variables $\{C_t:t>0\}$ that converges (at least in 
probability) to the origin $\mathbf{0}\in\mathbb{R}^h$, and satisfies the large 
deviation principle with speed $v_t\to\infty$ and rate function $I_{\mathrm{LD}}$; moreover 
$\{\sqrt{v_t}C_t:t>0\}$ converges weakly to the centered Normal distribution with covariance matrix 
$\Sigma$. Then, for every family of positive numbers $\{a_t:t>0\}$ such that
\begin{equation}\label{eq:MDconditions}
a_t\to 0\ \mbox{and}\ v_ta_t\to\infty,
\end{equation}
the family of random variables $\{\sqrt{a_tv_t}C_t:t>0\}$ satisfies the large deviation principle with 
speed $1/a_t$ and a rate function $I_{\mathrm{MD}}$ defined by
$$I_{\mathrm{MD}}(x):=\sup_{\theta\in\mathbb{R}^h}\left\{\langle\theta,x\rangle
-\frac{1}{2}\langle\theta,\Sigma\theta\rangle\right\}
\quad\mbox{for all}\ x\in\mathbb{R}^h;$$
moreover we typically have $I_{\mathrm{LD}}(x)=I_{\mathrm{MD}}(x)=0$ if and only if $x=\mathbf{0}$, and 
$I_{\mathrm{MD}}$ behaves as $I_{\mathrm{LD}}$ locally around $\mathbf{0}$.
\end{claim}

Here we recall a well-known prototype example with a discrete parameter $n$ in place of $t$. We set
$$C_n:=\frac{X_1+\cdots X_n}{n}\quad\mbox{for all}\ n\geq 1,$$
where $\{X_n:n\geq 1\}$ is a sequence of i.i.d. $\mathbb{R}^h$-valued centered random variables 
with finite covariance matrix $\Sigma$. For simplicity we also assume that 
$\mathbb{E}\left[e^{\langle\theta,X_1\rangle}\right]$ is finite if $\theta$ belongs to a neighborhood 
of $\mathbf{0}$. Then $\{C_n:n\geq 1\}$ converges to $\mathbf{0}$ by the law of large numbers, and 
satisfies the large deviation principle with speed $v_n=n$ and rate function $I_{\mathrm{LD}}$ defined by
$$I_{\mathrm{LD}}(x):=\sup_{\theta\in\mathbb{R}^h}\left\{\langle\theta,x\rangle
-\log\mathbb{E}\left[e^{\langle\theta,X_1\rangle}\right]\right\}
\quad\mbox{for all}\ x\in\mathbb{R}^h$$
(see e.g. Cram\'er Theorem, i.e. Theorem 2.2.30 in \cite{DemboZeitouni}). The weak convergence of
$\{\sqrt{n}C_n:n\geq 1\}$ is a consequence the central limit theorem. Finally, for every sequence
of positive numbers $\{a_n:n\geq 1\}$ such that \eqref{eq:MDconditions} holds with $v_n=n$
(actually we mean a version of \eqref{eq:MDconditions} for a discrete index $n$ in place of $t$), 
$\{\sqrt{a_nn}C_n:n\geq 1\}$ satisfies the LDP with rate function $I_{\mathrm{MD}}$ defined above
by Theorem 3.7.1 in \cite{DemboZeitouni}. Moreover, if $\Sigma$ is invertible, we have
$$I_{\mathrm{MD}}(x)=\frac{1}{2}\langle x,\Sigma^{-1}x\rangle$$
and the Hessian matrix of $I_{\mathrm{MD}}(x)$ at $x=\mathbf{0}$ is equal to $\Sigma^{-1}$.

We talk about non-central moderate deviations when we have a situation similar to the one in Claim 
\ref{claim:standardMD}, and the weak convergence is towards a non-Gaussian distribution.
Some univariate examples are presented in \cite{GiulianoMacci} and in some references
cited therein. An example with multivariate random variables can be found in 
\cite{LeonenkoMacciPacchiarotti} (see Section 3), even if in that case the weak convergence is 
trivial because one has a family of identically distributed random variables.

In this paper we consider a compound fractional Poisson process $\{S_{\nu,\lambda}(t):t\geq 0\}$ with 
light tailed jumps described in the next Condition \ref{cond:cfpp}.

\begin{condition}\label{cond:cfpp}
Let $\{S_{\nu,\lambda}(t):t\geq 0\}$ be defined by
$$S_{\nu,\lambda}(t):=\sum_{k=1}^{N_{\nu,\lambda}(t)}X_k$$
where $\{X_n:n\geq 1\}$ is a sequence of i.i.d. \emph{real} random variables such that 
$\mathbb{E}\left[e^{\theta X_1}\right]$ is finite if $\theta$ belongs to a neighborhood of $0$ 
(i.e. the light tail case), and $\{N_{\nu,\lambda}(t):t\geq 0\}$ is a time fractional 
Poisson process with $\nu\in(0,1)$, independent of $\{X_n:n\geq 1\}$. In particular the 
random variables $\{X_n:n\geq 1\}$ have (common) finite mean and variance, and we set
$$\mu:=\mathbb{E}[X_1]\ \mbox{and}\ \sigma^2:=\mathrm{Var}[X_1].$$
\end{condition}

There are several references on fractional Poisson process; here we recall \cite{BeghinOrsingher2009},
\cite{BeghinOrsingher2010} and \cite{MeerschaertNaneVellaisamy}; in particular we recall that we refer to 
the time fractional Poisson process (for the space and space-time fractional Poisson process see
\cite{OrsingherPolito}). Some properties of the process $\{N_{\nu,\lambda}(t):t\geq 0\}$ will be recalled 
in Section \ref{sec:preliminaries}. 

Our aim is to prove non-central moderate deviations for $C_t:=\frac{S_{\nu,\lambda}(t)}{t}$.
More precisely we bear in mind what we said in Claim \ref{claim:standardMD}, and we mean
the following three statements.
\begin{itemize}
\item The family of random variables $\left\{\frac{S_{\nu,\lambda}(t)}{t}:t>0\right\}$ satisfies the large deviation 
principle with speed $v_t=t$ and a rate function $I_{\mathrm{LD}}$ which does not depend on $\mu$ (see Proposition
\ref{prop:LD}). Note that we have $I_{\mathrm{LD}}(x)=0$ if and only if $x=0$; so $\frac{S_{\nu,\lambda}(t)}{t}$ 
converges to zero (as $t\to\infty$) at least in probability.
\item For
\begin{equation}\label{eq:exponents}
\alpha(\nu):=\left\{\begin{array}{ll}
1-\nu/2&\ \mbox{if}\ \mu=0\\
1-\nu&\ \mbox{if}\ \mu\neq 0,
\end{array}\right.
\end{equation}
$\left\{\frac{t^{\alpha(\nu)}S_{\nu,\lambda}(t)}{t}:t>0\right\}$ converges weakly toward some non-degenerate and 
non-Gaussian distribution (see Proposition \ref{prop:weak-convergence}).
\item For every sequence of positive numbers $\{a_t:t>0\}$ such that \eqref{eq:MDconditions} holds, the family of 
random variables $\left\{\frac{(a_tt)^{\alpha(\nu)}S_{\nu,\lambda}(t)}{t}:t>0\right\}$ satisfies the large deviation 
principle with speed $1/a_t$ and a rate function $I_{\mathrm{MD}}$, which uniquely vanishes at zero (see Proposition
\ref{prop:ncMD}).
\end{itemize}

So, in some sense, we have two non-central moderate deviation results concerning the cases $\mu=0$ and 
$\mu\neq 0$; however these two results share a common underlying large deviation principle for the convergence
in probability to zero which does not depend on $\mu$.

We conclude with the outline of the paper. In Section \ref{sec:preliminaries} we recall some preliminaries
on large deviations and on (possibly compound) fractional Poisson process. In Section \ref{sec:results} we
prove the results. We conclude with a brief discussion of the heavy tail case (we refer to the 
terminology in Condition \ref{cond:cfpp}) in Section \ref{sec:heavy-tail}.

\section{Preliminaries}\label{sec:preliminaries}
In this section we recall some preliminaries on large deviations and on fractional processes.

\subsection{On large deviations}
We start with the definition of large deviation principle (see e.g. \cite{DemboZeitouni}, pages
4-5). In view of what follows we present definitions and results for families of real random variables
$\{Z_t:t>0\}$ defined on the same probability space $(\Omega,\mathcal{F},P)$, and we consider $t\to\infty$.
A family of numbers $\{v_t:t>0\}$ such that $v_t\to\infty$ (as $t\to\infty$) is called a \emph{speed function}, 
and a lower semicontinuous function $I:\mathbb{R}\to[0,\infty]$ is called a \emph{rate function}. Then 
$\{Z_t:t>0\}$ satisfies the large deviation principle (LDP from now on) with speed $v_t$ and a rate function $I$
if
$$\limsup_{t\to\infty}\frac{1}{v_t}\log P(Z_t\in C)\leq-\inf_{x\in C}I(x)\quad\mbox{for all closed sets}\ C,$$
and
$$\liminf_{t\to\infty}\frac{1}{v_t}\log P(Z_t\in O)\geq-\inf_{x\in O}I(x)\quad\mbox{for all open sets}\ O.$$
The rate function $I$ is said to be \emph{good} if, for every $\beta\geq 0$, the level set 
$\{x\in\mathbb{R}:I(x)\leq\beta\}$ is compact. We also recall the following known result (see e.g. Theorem
2.3.6(c) in \cite{DemboZeitouni}).

\begin{theorem}[G\"artner Ellis Theorem]\label{th:GE}
Assume that, for all $\theta\in\mathbb{R}$, there exists
$$\Lambda(\theta):=\lim_{t\to\infty}\frac{1}{v_t}\log\mathbb{E}\left[e^{v_t\theta Z_t}\right]$$
as an extended real number; moreover assume that the origin $\theta=0$ 
belongs to the interior of the set
$$\mathcal{D}(\Lambda):=\{\theta\in\mathbb{R}:\Lambda(\theta)<\infty\}.$$
Furthermore let $\Lambda^*$ be the function defined by
$$\Lambda^*(x):=\sup_{\theta\in\mathbb{R}}\{\theta x-\Lambda(\theta)\}.$$
Then, if $\Lambda$ is essentially smooth and lower semi-continuous, then $\{Z_t:t>0\}$
satisfies the LDP with good rate function $\Lambda^*$.
\end{theorem}

We also recall (see e.g. Definition 2.3.5 in \cite{DemboZeitouni}) that $\Lambda$ is essentially smooth
if the interior of $\mathcal{D}(\Lambda)$ is non-empty, the function $\Lambda$ is differentiable 
throughout the interior of $\mathcal{D}(\Lambda)$, and $\Lambda$ is steep, i.e. $|\Lambda^\prime(\theta_n)|\to\infty$
whenever $\theta_n$ is a sequence of points in the interior of $\mathcal{D}(\Lambda)$ which converge to 
a boundary point of $\mathcal{D}(\Lambda)$.

\subsection{On (possibly compound) fractional Poisson process}
We start with the definition of the Mittag-Leffler function (see e.g. \cite{GorenfloKilbasMainardiRogosin}, eq. 
(3.1.1))
$$E_\nu(x):=\sum_{k=0}^\infty\frac{x^k}{\Gamma(\nu k+1)}.$$
It is known (see Proposition 3.6 in \cite{GorenfloKilbasMainardiRogosin} for the case $\alpha\in(0,2)$; indeed 
$\alpha$ in that reference coincides with $\nu$ in this paper) that we have
\begin{equation}\label{eq:ML-asymptotics}
\left\{\begin{array}{l}
E_\nu(x)\sim\frac{e^{x^{1/\nu}}}{\nu}\ \mbox{as}\ x\to\infty\\
E_\nu(x)\to 0\ \mbox{as}\ x\to-\infty.
\end{array}\right.
\end{equation}

Now we recall some moment generating functions which can be expressed in terms of the Mittag-Leffler
function. If we consider the inverse of the stable subordinator $\{L_\nu(t):t\geq 0\}$, then we have 
\begin{equation}\label{eq:MGF-inverse-stable-sub}
\mathbb{E}[e^{\theta L_\nu(t)}]=E_\nu(\theta t^\nu)\ \mbox{for all}\ \theta\in\mathbb{R}.
\end{equation}
This formula appears in several references with $\theta\leq 0$ only; however this restriction is not
needed because we can refer to the analytic continuation of the Laplace transform with complex argument.

Moreover the fractional process $\{N_{\nu,\lambda}(t):t\geq 0\}$ can be expressed as 
$$N_{\nu,\lambda}(t)=N_{1,\lambda}(L_\nu(t))\ \mbox{for all}\ t\geq 0,$$
i.e. a time-changed standard Poisson process $\{N_{1,\lambda}(t):t\geq 0\}$ with an independent inverse 
of the stable subordinator $\{L_\nu(t):t\geq 0\}$ (see e.g. Theorem 2.2 in \cite{MeerschaertNaneVellaisamy};
see also Remark 2.3 in the same article for other references with related results). Then it is easy to 
check that
$$\mathbb{E}[e^{\theta N_{\nu,\lambda}(t)}]=E_\nu(\lambda(e^\theta-1)t^\nu)\ \mbox{for all}\ \theta\in\mathbb{R}$$
and, moreover,
\begin{equation}\label{eq:MGF-S}
\mathbb{E}[e^{\theta S_{\nu,\lambda}(t)}]=E_\nu(\lambda(\mathbb{E}[e^{\theta X_1}]-1)t^\nu)\ \mbox{for all}\ \theta\in\mathbb{R}.
\end{equation}

\section{Results}\label{sec:results}
We start with the common underlying LDP for the convergence in probability of $\left\{\frac{S_{\nu,\lambda}(t)}{t}:t>0\right\}$
to zero which does not depend on $\mu$.

\begin{proposition}\label{prop:LD}
Assume that Condition \ref{cond:cfpp} holds. Moreover let $\Lambda_{\nu,\lambda}$
be the function defined by
\begin{equation}\label{eq:LD-GE-limit}
\Lambda_{\nu,\lambda}(\theta):=\left\{\begin{array}{ll}
(\lambda(\mathbb{E}[e^{\theta X_1}]-1))^{1/\nu}&\ \mbox{if}\ \mathbb{E}[e^{\theta X_1}]>1\\
0&\ \mbox{if}\ \mathbb{E}[e^{\theta X_1}]\leq 1,
\end{array}\right.
\end{equation}
and assume that it is essentially smooth. Then $\left\{\frac{S_{\nu,\lambda}(t)}{t}:t>0\right\}$ 
satisfies the LDP with speed $v_t=t$ and good rate function $I_{\mathrm{LD}}$ defined by
\begin{equation}\label{eq:LD-rf}
I_{\mathrm{LD}}(x):=\sup_{\theta\in\mathbb{R}}\{\theta x-\Lambda_{\nu,\lambda}(\theta)\}.
\end{equation}
\end{proposition}
\begin{proof}
The desired LDP can be derived by applying the G\"artner Ellis Theorem (i.e. Theorem \ref{th:GE}); in fact we have
$$\lim_{t\to\infty}\frac{1}{t}\log\mathbb{E}[e^{\theta S_{\nu,\lambda}(t)}]=\Lambda_{\nu,\lambda}(\theta)
\ \mbox{for all}\ \theta\in\mathbb{R}$$
by \eqref{eq:MGF-S} and \eqref{eq:ML-asymptotics}.
\end{proof}

\begin{remark}\label{rem:difference-with-the-case-nu=1}
In Proposition \ref{prop:LD}, since $\nu\in(0,1)$, we have $I_{\mathrm{LD}}=0$ if and only if 
$x=\Lambda_{\nu,\lambda}^\prime(0)=0$ for every $\mu\in\mathbb{R}$. On the other hand, if $\nu=1$ (and if the 
function $\Upsilon$ defined by
$$\Upsilon_\lambda(\theta):=\left\{\begin{array}{ll}
\lambda(\mathbb{E}[e^{\theta X_1}]-1)&\ \mbox{if}\ \mathbb{E}[e^{\theta X_1}]<\infty\\
0&\ \mbox{otherwise}
\end{array}\right.$$
is essentially smooth), it is well-known that $\left\{\frac{S_{\nu,\lambda}(t)}{t}:t>0\right\}$ satisfies the LDP 
with speed $v_t=t$ and good rate function $I_{\mathrm{LD}}$ defined by 
$$I_{\mathrm{LD}}(x):=\sup_{\theta\in\mathbb{R}}\{\theta x-\Upsilon_\lambda(\theta)\}.$$
In such a case we have $I_{\mathrm{LD}}(x)=0$ if and only if $\Upsilon_\lambda^\prime(0)=\lambda\mu$.
\end{remark}

\begin{remark}\label{rem:link-with-SPL2013-paper}
If we consider the case $X_n=1$ for all $n\geq 1$, Proposition \ref{prop:LD} yields the LDP for
$\left\{\frac{N_{\nu,\lambda}(t)}{t}:t>0\right\}$ and the rate function is given by \eqref{eq:LD-rf} with
and $\Lambda_{\nu,\lambda}$ in \eqref{eq:LD-GE-limit} reads
$$\Lambda_{\nu,\lambda}(\theta):=\left\{\begin{array}{ll}
(\lambda(e^\theta-1))^{1/\nu}&\ \mbox{if}\ \theta>0\\
0&\ \mbox{if}\ \theta\leq 0.
\end{array}\right.$$
So one can check that
$$I_{\mathrm{LD}}(x):=\left\{\begin{array}{ll}
\sup_{\theta>0}\{\theta x-(\lambda(e^\theta-1))^{1/\nu}\}&\ \mbox{if}\ x>0\\
0&\ \mbox{if}\ x=0\\
\infty&\ \mbox{if}\ x<0.
\end{array}\right.$$
The LDP for $\left\{\frac{N_{\nu,\lambda}(t)}{t}:t>0\right\}$ was already proved; see Propositions 3.1 and
3.2 in \cite{BeghinMacciSPL2013} with $h=1$, where the rate function expression is slightly different,
i.e.
$$I_{\mathrm{LD}}(x):=\left\{\begin{array}{ll}
x\sup_{\eta<0}\left\{\frac{\eta}{x}-\log\frac{\lambda}{\lambda+(-\eta)^\nu}\right\}&\ \mbox{if}\ x>0\\
0&\ \mbox{if}\ x=0\\
\infty&\ \mbox{if}\ x<0.
\end{array}\right.$$
Actually the rate functions expressions coincide; in fact, if we set $\theta=-\log\frac{\lambda}{\lambda+(-\eta)^\nu}$,
for $x>0$ we have
$$x\sup_{\eta<0}\left\{\frac{\eta}{x}-\log\frac{\lambda}{\lambda+(-\eta)^\nu}\right\}
=\sup_{\theta>0}\{-(\lambda(e^\theta-1))^{1/\nu}+x\theta\}.$$
\end{remark}

\begin{remark}\label{rem:Glynn-Whitt}
We recall that $\{N_{\nu,\lambda}(t):t\geq 0\}$ is a renewal process; so we have
$$N_{\nu,\lambda}(t):=\sum_{k=1}^\infty1_{\{T_1+\cdots+T_n\leq t\}}$$
for some the i.i.d. interarrival times $\{T_n:n\geq 1\}$. Then, if we set
$$\kappa(\eta):=\log\mathbb{E}[e^{\eta T_1}],$$
in Remark \ref{rem:link-with-SPL2013-paper} we have considered the equality 
$\theta=-\kappa(\eta)$ for $\eta\in(0,\infty)$. In conclusion we have
$$x\sup_{\eta<0}\left\{\frac{\eta}{x}-\kappa(\eta)\right\}
=\sup_{\theta>0}\{\kappa^{-1}(-\theta)+x\theta\}=:\Psi_\kappa^*(x),$$
where
$$\Psi_\kappa^*(x):=\sup_{\theta>0}\{\theta x-\Psi_\kappa(\theta)\}\ \mbox{and}\ \Psi_\kappa(\theta):=-\kappa^{-1}(-\theta),$$
and this agrees with formulas (12)-(13) in \cite{GlynnWhitt}.
\end{remark}

Now we present the weak convergence results as $t\to\infty$. For the sake of completeness we give a brief
proof by taking the limit of the moment generating functions even if some of these results are known.
For instance the convergence for $\mu\neq 0$ agrees with the weak convergence stated just after eq. (3.7) 
in \cite{VellaisamyMaheshwari} for a less general case (i.e. for the case $X_n=1$ for all $n\geq 1$, and 
therefore for $\{N_{\nu,\lambda}(t):t\geq 0\}$ instead of $\{S_{\nu,\lambda}(t):t\geq 0\}$). The convergence
for $\mu=0$ appears in \cite{Meerschaert-et-al} (Section II), in \cite{MeerschaertScheffler} (Theorem 4.2)
and it is also cited in the introduction of \cite{LeonenkoMeerschaertSchillingSikorskii}; however in those 
references the results are given for sample paths.
Another recent weak convergence result with $\mu=0$ appears in \cite{OliveiraBarretosouzaSilva} (Proposition 2.1);
they let $\lambda$ go to infinity with $t=1$, and they get the same limit distribution called 
\emph{Normal variance mixture}.

\begin{proposition}\label{prop:weak-convergence}
Assume that Condition \ref{cond:cfpp} holds and let $\alpha(\nu)$ be defined in \eqref{eq:exponents}.
Then:
\begin{itemize}
\item if $\mu=0$, then $\{t^{\alpha(\nu)}\frac{S_{\nu,\lambda}(t)}{t}:t>0\}$ converges weakly to
$\sqrt{\lambda\sigma^2L_\nu(1)}Z$, where $Z$ is a standard Normal distributed random variable, and
independent to $L_\nu(1)$;
\item if $\mu\neq 0$, then $\{t^{\alpha(\nu)}\frac{S_{\nu,\lambda}(t)}{t}:t>0\}$ converges weakly to
$\lambda\mu L_\nu(1)$.
\end{itemize}
\end{proposition}
\begin{proof}
In both cases $\mu=0$ and $\mu\neq 0$ we study the limit as $t\to\infty$ of the moment generating 
functions. We take into account \eqref{eq:MGF-S} for the expressions of the moment generating functions,
and we take into account \eqref{eq:ML-asymptotics} when we take the limit.

If $\mu=0$ we have
\begin{multline*}
\mathbb{E}\left[e^{\theta t^{\alpha(\nu)}\frac{S_{\nu,\lambda}(t)}{t}}\right]
=\mathbb{E}\left[e^{\theta\frac{S_{\nu,\lambda}(t)}{t^{\nu/2}}}\right]
=E_\nu(\lambda(\mathbb{E}[e^{\theta X_1/t^{\nu/2}}]-1)t^\nu)\\
=E_\nu\left(\lambda\left(1+\frac{\sigma^2\theta^2}{2t^\nu}+o\left(\frac{1}{t^\nu}\right)-1\right)t^\nu\right)
\to E_\nu\left(\lambda\frac{\sigma^2\theta^2}{2}\right)\ \mbox{for all}\ \theta\in\mathbb{R}.
\end{multline*}
Thus the desired weak convergence is proved noting that (here we take into account 
\eqref{eq:MGF-inverse-stable-sub})
$$\mathbb{E}\left[e^{\theta\sqrt{\lambda\sigma^2L_\nu(1)}Z}\right]
=\mathbb{E}\left[e^{\frac{\theta^2\lambda\sigma^2}{2}L_\nu(1)}\right]
=E_\nu\left(\lambda\frac{\sigma^2\theta^2}{2}\right)\ \mbox{for all}\ \theta\in\mathbb{R}.$$

If $\mu\neq 0$ we have
\begin{multline*}
\mathbb{E}\left[e^{\theta t^{\alpha(\nu)}\frac{S_{\nu,\lambda}(t)}{t}}\right]
=\mathbb{E}\left[e^{\theta\frac{S_{\nu,\lambda}(t)}{t^\nu}}\right]
=E_\nu(\lambda(\mathbb{E}[e^{\theta X_1/t^\nu}]-1)t^\nu)\\
=E_\nu\left(\lambda\left(1+\frac{\mu\theta}{t^\nu}+\frac{\sigma^2\theta^2}{2t^{2\nu}}+o\left(\frac{1}{t^{2\nu}}\right)-1\right)t^\nu\right)
\to E_\nu\left(\lambda\mu\theta\right)\ \mbox{for all}\ \theta\in\mathbb{R}.
\end{multline*}
Thus the desired weak convergence is proved by \eqref{eq:MGF-inverse-stable-sub}.
\end{proof}

Now we present the non-central moderate deviation results.

\begin{proposition}\label{prop:ncMD}
Assume that Condition \ref{cond:cfpp} holds and let $\alpha(\nu)$ be defined in \eqref{eq:exponents}.
Then, for every family of positive numbers $\{a_t:t>0\}$ such that \eqref{eq:MDconditions} holds, the
family of random variables $\left\{\frac{(a_tt)^{\alpha(\nu)}S_{\nu,\lambda}(t)}{t}:t>0\right\}$ 
satisfies the LDP with speed $1/a_t$ and good rate function $I_{\mathrm{MD},\mu}$ defined by:
$$\left.\begin{array}{cc}
\mbox{if}\ \mu=0,&\ I_{\mathrm{MD},\mu}(x):=((\nu/2)^{\nu/(2-\nu)}-(\nu/2)^{2/(2-\nu)})\left(\frac{2x^2}{\lambda\sigma^2}\right)^{1/(2-\nu)};\\
\mbox{if}\ \mu>0,&\ I_{\mathrm{MD},\mu}(x):=\left\{\begin{array}{ll}
(\nu^{\nu/(1-\nu)}-\nu^{1/(1-\nu)})\left(\frac{x}{\lambda\mu}\right)^{1/(1-\nu)}&\ \mbox{if}\ x\geq 0\\
\infty&\ \mbox{if}\ x<0;
\end{array}\right.\\
\mbox{if}\ \mu<0,&\ I_{\mathrm{MD},\mu}(x):=\left\{\begin{array}{ll}
(\nu^{\nu/(1-\nu)}-\nu^{1/(1-\nu)})\left(\frac{x}{\lambda\mu}\right)^{1/(1-\nu)}&\ \mbox{if}\ x\leq 0\\
\infty&\ \mbox{if}\ x>0.
\end{array}\right.
\end{array}\right.$$
\end{proposition}
\begin{proof}
For every $\mu\in\mathbb{R}$ we apply the G\"artner Ellis Theorem (Theorem \ref{th:GE}). So we have to take
$$\Lambda_{\nu,\lambda,\mu}(\theta):=\lim_{t\to\infty}
\frac{1}{1/a_t}\log\mathbb{E}\left[e^{\frac{\theta}{a_t}\frac{(a_tt)^{\alpha(\nu)}S_{\nu,\lambda}(t)}{t}}\right]
\ \mbox{for all}\ \theta\in\mathbb{R},$$
or equivalently
$$\Lambda_{\nu,\lambda,\mu}(\theta)
:=\lim_{t\to\infty}a_t\log E_\nu\left(\lambda\left(\mathbb{E}\left[e^{\frac{\theta}{(a_tt)^{1-\alpha(\nu)}}X_1}\right]-1\right)\right)
\ \mbox{for all}\ \theta\in\mathbb{R};$$
in particular we refer to \eqref{eq:ML-asymptotics} when we take the limit. Moreover, for 
every $\mu$, the function $\Lambda_{\nu,\lambda,\mu}$ satisfies the hypotheses of the G\"artner Ellis Theorem 
(this can be checked by considering the expressions below), and therefore the LDP holds with good rate
function $I_{\mathrm{MD},\mu}$ defined by
\begin{equation}\label{eq:MD-Legendre-transform}
I_{\mathrm{MD},\mu}(x):=\sup_{\theta\in\mathbb{R}}\{\theta x-\Lambda_{\nu,\lambda,\mu}(\theta)\}.
\end{equation}
Then, as we shall explain below, for every $\mu$ the rate function expression in \eqref{eq:MD-Legendre-transform}
coincides with the rate function $I_{\mathrm{MD},\mu}$ in the statement.

If $\mu=0$ we have
\begin{multline*}
a_t\log E_\nu\left(\lambda\left(\mathbb{E}\left[e^{\frac{\theta}{(a_tt)^{1-\alpha(\nu)}}X_1}\right]-1\right)\right)
=a_t\log E_\nu\left(\lambda\left(1+\frac{\theta^2\sigma^2}{2(a_tt)^\nu}+o\left(\frac{1}{(a_tt)^\nu}\right)-1\right)t^\nu\right)\\
=a_t\log E_\nu\left(\lambda\left(\frac{\theta^2\sigma^2}{2(a_tt)^\nu}+o\left(\frac{1}{(a_tt)^\nu}\right)\right)t^\nu\right)
=a_t\log E_\nu\left(\frac{\lambda}{a_t^\nu}\left(\frac{\theta^2\sigma^2}{2}+(a_tt)^\nu o\left(\frac{1}{(a_tt)^\nu}\right)\right)\right),
\end{multline*}
and therefore
$$\lim_{t\to\infty}a_t\log E_\nu\left(\lambda\left(\mathbb{E}\left[e^{\frac{\theta}{(a_tt)^{1-\alpha(\nu)}}X_1}\right]-1\right)\right)
=\left(\frac{\lambda\theta^2\sigma^2}{2}\right)^{1/\nu}=:\Lambda_{\nu,\lambda,\mu}(\theta)\ \mbox{for all}\ \theta\in\mathbb{R};$$
thus the desired LDP holds with good rate function $I_{\mathrm{MD},\mu}$ defined by \eqref{eq:MD-Legendre-transform}
which coincides with the rate function expression in the statement (indeed one can check that this supremum is attained at 
$\theta=\theta_x:=\left(\frac{2}{\lambda\sigma^2}\right)^{1/(2-\nu)}\left(\frac{\nu x}{2}\right)^{\nu/(2-\nu)}$).

If $\mu>0$ we have
\begin{multline*}
a_t\log E_\nu\left(\lambda\left(\mathbb{E}\left[e^{\frac{\theta}{(a_tt)^{1-\alpha(\nu)}}X_1}\right]-1\right)\right)\\
=a_t\log E_\nu\left(\lambda\left(1+\frac{\theta\mu}{(a_tt)^\nu}+o\left(\frac{1}{(a_tt)^\nu}\right)-1\right)t^\nu\right)\\
=a_t\log E_\nu\left(\frac{\lambda}{a_t^\nu}\left(\theta\mu+(a_tt)^\nu o\left(\frac{1}{(a_tt)^\nu}\right)\right)\right)
\end{multline*}
and therefore
$$\lim_{t\to\infty}a_t\log E_\nu\left(\lambda\left(\mathbb{E}\left[e^{\frac{\theta}{(a_tt)^{1-\alpha(\nu)}}X_1}\right]-1\right)\right)
=\left\{\begin{array}{ll}
(\lambda\theta\mu)^{1/\nu}&\ \mbox{if}\ \theta>0\\
0&\ \mbox{if}\ \theta\leq 0
\end{array}\right.=:\Lambda_{\nu,\lambda,\mu}(\theta)\ \mbox{for all}\ \theta\in\mathbb{R};$$
thus the desired LDP holds with good rate function $I_{\mathrm{MD},\mu}$ defined by \eqref{eq:MD-Legendre-transform}
which coincides with the rate function expression in the statement (indeed one can check that this supremum is equal to 
infinity for $x<0$ (by letting $\theta$ go to $-\infty$), and it is attained at 
$\theta=\theta_x:=\frac{(\nu x)^{\nu/(1-\nu)}}{(\lambda\mu)^{1/(1-\nu)}}$ for $x\geq 0$).

If $\mu<0$ we can repeat the same computations presented for the case $\mu>0$ but, when we take the limit as $t\to\infty$, 
we have
$$\Lambda_{\nu,\lambda,\mu}(\theta):=\left\{\begin{array}{ll}
(\lambda\theta\mu)^{1/\nu}&\ \mbox{if}\ \theta<0\\
0&\ \mbox{if}\ \theta\geq 0
\end{array}\right.\ \mbox{for all}\ \theta\in\mathbb{R};$$
thus the desired LDP holds with good rate function $I_{\mathrm{MD},\mu}$ defined by \eqref{eq:MD-Legendre-transform}
which coincides with the rate function expression in the statement (indeed one can check that this supremum is equal to 
infinity for $x>0$ (by letting $\theta$ go to $+\infty$), and it is attained at
$\theta=\theta_x:=-\frac{(-\nu x)^{\nu/(1-\nu)}}{(-\lambda\mu)^{1/(1-\nu)}}$ for $x\leq 0$).
\end{proof}

\begin{remark}\label{rem:*}
The rate function $I_{\mathrm{MD},\mu}$ for the case $\mu<0$ can be computed by referring to the case $\mu>0$; indeed
we have
$$I_{\mathrm{MD},\mu}(x)=I_{\mathrm{MD},-\mu}(-x)=\left\{\begin{array}{ll}
(\nu^{\nu/(1-\nu)}-\nu^{1/(1-\nu)})\left(\frac{(-x)}{\lambda(-\mu)}\right)^{1/(1-\nu)}&\ \mbox{if}\ -x\geq 0\\
\infty&\ \mbox{if}\ -x<0,
\end{array}\right.$$
and we immediately recover the rate function in the statement of Proposition \ref{prop:ncMD} for the case $\mu<0$.
\end{remark}

\section{A brief discussion on the heavy tail case}\label{sec:heavy-tail}
One can wonder what happens if the random variables $\{X_n:n\geq 1\}$ are not light-tailed distributed (see Condition
\ref{cond:cfpp}). Obviously in this case we cannot apply the G\"artner Ellis Theorem, and we cannot say how to approach
the problem. 

Typically the results on the asymptotic behavior of (possibly compound) sums of i.i.d. heavy tailed distributed random 
variables are not formulated in terms of large deviation principles. However some references provide large deviation 
principles for sums of i.i.d. semi-exponential distributed random variables; in particular the heavy tailed Weibull 
distribution, i.e. the case with distribution function
$$F(x)=1-e^{-(\lambda x)^r}\ \mbox{for}\ x\geq 0),\ \mbox{for some}\ r\in(0,1),$$
belongs to this class. Lemma 1 in \cite{Gantert} provides the LDP for empirical means (that reference also provides
sample path versions of this result); the speed function is $n^r$ (so it is a slower speed) and the rate function is 
not convex. In some sense the results in \cite{Gantert} reveal the bigger influence of extreme values on the partial sums,
and the situation is very different from the light tail case.

Large deviations for compound Poisson sums of i.i.d. semi-exponential distributed random variables can be derived from 
Proposition 2.1 in \cite{StabileTorrisi} (a sample-path version of this result can be found \cite{DuffyTorrisi}) which 
concerns Poisson shot noise processes, and the result does not depend on the shot shape. The rate function keeps the 
properties cited above for the result in \cite{Gantert} case without compound sums. So one could expect that it is possible
to prove a similar result for the compound fractional Poisson processes in this paper when $\{X_n:n\geq 1\}$ are i.i.d. 
and semi-exponential distributed. In this case we should have the analogue of Proposition \ref{prop:LD} in this paper, and
this could be a starting point to obtain a non-central moderate deviations for a heavy tail case.

\paragraph{Acknowledgements.}
We thank the referee for suggesting the idea to discuss the heavy tail case.


\begin{thebibliography}{spc}
\bibitem{Billingsley}
P. Billingsley (1995) Probability and Measure. 3nd Edition, Wiley, New York.
\bibitem{BeghinMacciSPL2013}
L. Beghin, C. Macci (2013) Large deviations for fractional Poisson 
processes. Statist. Probab. Lett. 83, 1193--1202.
\bibitem{BeghinOrsingher2009}
L. Beghin, E. Orsingher (2009) Fractional Poisson processes and related planar
random motions. Electron. J. Probab. 14, 1790--1827.
\bibitem{BeghinOrsingher2010}
L. Beghin, E. Orsingher (2010) Poisson-type processes governed by fractional 
and higher-order recursive differential equations. Electron. J. Probab. 15,
684--709.
\bibitem{DemboZeitouni}
A. Dembo, O. Zeitouni (1998) Large Deviations Techniques and Applications (Second 
Edition). Springer-Verlag, New York.
\bibitem{DuffyTorrisi}
K.R. Duffy, G.L. Torrisi (2011) Sample path large deviations of Poisson shot noise 
with heavy-tailed semiexponential distributions. J. Appl. Probab. 48, 688--698.
\bibitem{Gantert}
N. Gantert (1998) Functional Erd\"os-Renyi laws for semiexponential random variables. 
Ann. Probab. 26, 1356--1369.
\bibitem{GiulianoMacci}
R. Giuliano, C. Macci (2021+) Some examples of non-central moderate deviations
for sequences of real random variables. Available on https://arxiv.org/abs/2110.05859
\bibitem {GlynnWhitt}
P.W. Glynn, W. Whitt (1994) Large deviations behavior of counting 
processes and their inverses. Queueing Systems Theory Appl. 17, 
107--128.
\bibitem{GorenfloKilbasMainardiRogosin}
R. Gorenflo, A.A. Kilbas, F. Mainardi, S.V. Rogosin (2014) 
Mittag-Leffler Functions, Related Topics and Applications. Springer,
New York.
\bibitem{LeonenkoMacciPacchiarotti}
N. Leonenko, C. Macci, B. Pacchiarotti (2021) Large deviations for
a class of tempered subordinators and their inverse processes. Proc.
Roy. Soc. Edinburgh Sect. A 151, 2030--2050.
\bibitem{LeonenkoMeerschaertSchillingSikorskii}
N.N. Leonenko, M.M. Meerschaert, R.L. Schilling, A. Sikorskii (2014)
Correlation structure of time-changed L\'evy processes.
Commun. Appl. Ind. Math. 6, no. 1, e-483, 22 pp.
\bibitem{Meerschaert-et-al}
M.M. Meerschaert, D.A. Benson, H.P. Scheffler, B. Baeumer (2002) 
Stochastic solution of space-time fractional diffusion equations. 
Phys. Rev. E (3) 65, no. 4, paper 041103, 4 pp.
\bibitem{MeerschaertNaneVellaisamy}
M.M. Meerschaert, E. Nane, P. Vellaisamy (2011) The fractional Poisson
process and the inverse stable subordinator. Electron. J. Probab. 16, 
1600--1620. 
\bibitem{MeerschaertScheffler}
M.M. Meerschaert, H.P. Scheffler (2004) Limit theorems for continuous-time 
random walks with infinite mean waiting times. J. Appl. Probab. 41, 623--638.
\bibitem{OliveiraBarretosouzaSilva}
G. Oliveira, W. Barreto-Souza, R.W.C. Silva (2021+) Fractional 
Poisson random sum and its associated normal variance mixture.
To appear in Stoch. Models. Available at
https://www.tandfonline.com/doi/full/10.1080/15326349.2021.1954533
or at https://arxiv.org/pdf/2103.08691.pdf
\bibitem{OrsingherPolito}
E. Orsingher, F. Polito (2012) The space-fractional Poisson process. Statist.
Probab. Lett. 82 (2012), 852--858.
\bibitem{StabileTorrisi}
G. Stabile, G.L. Torrisi (2010) Large deviations of Poisson shot noise processes 
under heavy tail semi-exponential conditions. Statist. Probab. Lett. 80, 1200--1209.
\bibitem{VellaisamyMaheshwari}
P. Vellaisamy, A. Maheshwari (2018) Fractional negative binomial
and P\'olya processes. Probab. Math. Statist. 38, 77--101.
\end{thebibliography}
\end{document}